\documentclass[11pt]{amsart}

\usepackage[T1]{fontenc}
\usepackage{lmodern}
\usepackage{microtype}
\usepackage{amsmath,amssymb,mathtools}
\usepackage{enumitem}
\usepackage{xcolor}

\definecolor{linkblue}{RGB}{0,70,130}
\definecolor{citegreen}{RGB}{0,100,75}
\definecolor{urlblue}{RGB}{0,90,140}

\usepackage[
  colorlinks=true,
  linkcolor=linkblue,
  citecolor=citegreen,
  urlcolor=urlblue
]{hyperref}
\usepackage{orcidlink}
\usepackage{todonotes}
\makeatletter
\renewcommand{\@makefntext}[1]{%
  \noindent\@makefnmark\enspace#1%
}
\makeatother
\newtheoremstyle{compactplain}
  {3pt}        
  {3pt}        
  {\itshape}   
  {}           
  {\bfseries}  
  {.}          
  {0.5em}      
  {}           

\theoremstyle{compactplain}
\newtheorem{theorem}{Theorem}[section]
\newtheorem{proposition}[theorem]{Proposition}
\newtheorem{lemma}[theorem]{Lemma}

\theoremstyle{definition}
\newtheorem{definition}[theorem]{Definition}

\theoremstyle{remark}
\newtheorem{remark}[theorem]{Remark}

\DeclareMathOperator{\Prob}{Prob}
\DeclareMathOperator{\supp}{supp}

\newcommand{\E}{\mathbb E}
\newcommand{\one}{\mathbf 1}
\newcommand{\na}{\mathrm{na}}

\def\mc{\mathcal}
\def\mb{\mathbb}
\def\mr{\mathrm}
\def\M{\mc M}
\def\id{\mr{id}}

\title[Li--Yorke Chaos Along Any Infinite Sequence]
{Li--Yorke Chaos Along Any Infinite Sequence:
Relative Mixing, Sofic and Rokhlin Entropy versus Naive Entropy}

\author[C. Liu]{Chunlin Liu\orcidlink{0000-0001-6277-013X}}

\address[C. Liu]{School of Mathematical Sciences,
Dalian University of Technology,
Dalian 116024, P.R. China;
and Institute of Mathematics,
Polish Academy of Sciences,
ul. Śniadeckich 8,
00-656 Warszawa, Poland}

\email{chunlinliu@mail.ustc.edu.cn; chunlin.liu.math@gmail.com}
\thanks{This article  was supported by the
Postdoctoral Fellowship Program and China Postdoctoral Science Foundation under Grant Number BX20250067, and the China Postdoctoral Science Foundation under Grant Number 2025M773074.
}
\begin{document}
\begin{abstract}
Let $G$ be a countably infinite discrete group and let
$
\pi:(X,\mu,G)\to(Y,\nu,G)
$
be a nontrivial relatively mixing extension, where $X$ is a compact
metrizable $G$-space. We prove that there exists a constant $\delta>0$ such that, for every injective sequence
$(s_i)_{i\geq 1}$ in $G$, there is a Cantor set $K_{(s_i)}\subseteq X$
whose distinct points $x,x'$ satisfy
\[
\liminf_{i\to\infty}\rho(s_i x,s_i x')=0,
\qquad
\limsup_{i\to\infty}\rho(s_i x,s_i x')>\delta.
\]
The method also yields higher-order scrambled Cantor sets. As a principal application, for a sofic group $G$, positive
topological sofic entropy implies
the preceding conclusion, answering a question of Huang, Li, and Ye. The
same conclusion  also holds for
actions of arbitrary countably infinite discrete groups admitting an essentially free
invariant  measure of positive Rokhlin entropy.

In contrast, we construct a hereditary subshift over a residually finite group with positive naive topological entropy for which all points converge to a common fixed point along a prescribed injective sequence. This answers a question of Garc\'ia-Ramos and Li negatively even for sofic groups.
\end{abstract}
\maketitle
\section{Introduction}\label{sec:introduction}

Li--Yorke chaos was introduced in the seminal work of Li and Yorke
\cite{LiYorke}.  For a compact metric dynamical system $(X,T)$, a pair of
points $x\neq x'$ is called a \emph{Li--Yorke pair} if
\[
 \liminf_{n\to\infty}\rho(T^n x,T^n x')=0
 \qquad\text{and}\qquad
 \limsup_{n\to\infty}\rho(T^n x,T^n x')>0.
\]
A fundamental result of Blanchard, Glasner, Kolyada, and Maass shows that
positive topological entropy implies Li--Yorke chaos
\cite{BlanchardGlasnerKolyadaMaass}.  Subsequent work connected this
phenomenon with asymptotic pairs, relative entropy, and combinatorial
independence; see, among others,
\cite{BlanchardHostRuette,ZhangRelative,KerrLiIndependence}.  For actions of
countable groups, Kerr and Li proved that positive sofic topological entropy
implies Li--Yorke chaos in the groupwise sense
\cite[Corollary~8.4]{KerrLiCI}.

The situation becomes considerably more rigid when the observation times are
prescribed in advance.  Let $G$ act continuously on a compact metric space
$(X,\rho)$, and let $(s_i)_{i\geq1}$ be a sequence of pairwise distinct
elements of $G$.  One asks for an uncountable set whose distinct points
satisfy
\[
 \liminf_{i\to\infty}\rho(s_i x,s_i x')=0
 \qquad\text{and}\qquad
 \limsup_{i\to\infty}\rho(s_i x,s_i x')>0.
\]
This is stronger than ordinary groupwise Li--Yorke chaos: close and separated
orbit times available somewhere in $G$ need not survive after restricting to
an arbitrary sparse injective sequence.  Huang, Li, and Ye proved that every
positive-entropy action of a countably infinite amenable group is Li--Yorke
chaotic along every prescribed injective sequence
\cite[Theorem~1.1]{HuangLiYe}.  They asked whether the same conclusion holds for sofic group actions
\cite[Remark~4.3]{HuangLiYe}. Garc\'ia-Ramos and Li later asked
more generally whether positive naive topological entropy implies
Li--Yorke chaos along every infinite subset of the acting group
\cite[Question~5.28]{GarciaRamosLiLocalEntropy}. Our first main result
answers the former question affirmatively and, in view of the relation
between sofic and naive entropy, gives a partial positive answer to the
latter. We will see in Theorem~\ref{thm:naive-counterexample} that
positive naive topological entropy alone is not sufficient.

\begin{theorem}\label{thm:sofic}
Let $G$ be a countably infinite sofic group, let
$
 \Sigma=(\sigma_j:G\to\operatorname{Sym}(d_j))_{j\geq1}
$
be a fixed sofic approximation, and let $G\curvearrowright(X,\rho)$ be a
continuous action on a compact metrizable space.  If the sofic entropy
$
 h^{\mathrm{top}}_\Sigma(X,G)>0,
$
then there exists $\delta>0$ such that, for every sequence
$\mathbf{s}=(s_i)_{i\geq1}$ of pairwise distinct elements of $G$, there is a Cantor set
$
 K_{\mathbf{s}}\subseteq X
$
for which every distinct $x,x'\in K_{\mathbf{s}}$ satisfy
\[
 \liminf_{i\to\infty}\rho(s_i x,s_i x')=0,
 \qquad
 \limsup_{i\to\infty}\rho(s_i x,s_i x')>\delta.
\]
The constant $\delta$ is independent of the prescribed sequence
$\mathbf{s}$.
\end{theorem}

The proof of Theorem~\ref{thm:sofic} requires a mechanism different
from that used by Huang, Li, and Ye in the amenable setting
\cite{HuangLiYe}.  Starting from an ergodic measure of positive entropy,
they disintegrate over the Pinsker factor and work on the relatively
independent square.  The central input is the approximate additivity
formula for conditional entropy established in
\cite[Theorem~2.6]{HuangLiYe}.  This formula applies to finite subsets
$Q\subseteq G$ whose nonidentity pairwise quotients avoid a prescribed
finite subset of $G$.  By recursively extracting such sets $Q$ from the
tail of the prescribed sequence, and by using the Pinsker structure of
the relatively independent square, they obtain full conditional measure
for the required close-visit and separated-visit relations.

To the best of our knowledge, the available tools in sofic entropy
theory have not yet yielded an analogue of this conditional-entropy
statement in a form suitable for carrying out the same argument.
Sofic entropy is defined through finite model spaces rather than finite
joins over subsets of the acting group.  Moreover, although Pinsker
product formulae are known under additional hypotheses---notably
Hayes's outer Pinsker product formula for strongly sofic actions
\cite{HayesPinskerProduct}---these results do not appear to provide the
approximate conditional-entropy additivity required in
\cite[Theorem~2.6]{HuangLiYe}.  Huang, Li, and Ye themselves pointed
out this difficulty in \cite[Remark~4.3]{HuangLiYe}.

Our proof of Theorem~\ref{thm:sofic} replaces this
conditional-entropy argument with a measure-theoretic mechanism based
on relative mixing.  Relative mixing passes to the relatively
independent square and yields a relative Blum--Hanson averaging
principle along arbitrary injective sequences (Lemma \ref{lem:averaging}).  Together with the
Dirac--nonatomic dichotomy for the conditional measures (Lemma \ref{lem:dichotomy}) and a
fiberwise Mycielski argument (Subsection \ref{subsec:mycielski}), this produces the required close and
separated visits with a separation scale independent of the prescribed
sequence.  Since this mechanism is independent of entropy, it leads to
the following criterion for actions of arbitrary countably infinite
groups.
\begin{theorem}\label{thm:abstract}
Let $G\curvearrowright(X,\rho)$ be a continuous action on a compact
metrizable space preserving a Borel probability measure $\mu$, and let
$
 \pi:(X,\mu,G)\to(Y,\nu,G)
$
be a relatively mixing factor map, which is  not an isomorphism.  Write
$
 \mu=\int_Y\mu_y\,d\nu(y)
$
for the disintegration and put
\[
 Y_{\na}:=\{y\in Y:\mu_y\text{ is nonatomic}\}.
\]
Then $\nu( Y_{\na})>0$, and the following statements hold.
\begin{enumerate}[label=\textup{(\roman*)},leftmargin=2.2em]
\item
For every injective sequence $\mathbf{s}=(s_i)_{i\geq 1}\subset G$ and for  $\nu$-a.e. $y\in Y_{\na}$, there are
$\delta_{y,\mathbf{s}}>0$ and a dense Mycielski set\footnote{A subset $M$ of a metrizable space $X$ is called a \emph{Mycielski set}
if it can be written as a countable union of Cantor sets.}
$
 M_{y,\mathbf{s}}\subseteq\supp(\mu_y)
$
whose distinct points $x,x'$ satisfy
\[
 \liminf_{i\to\infty}\rho(s_i x,s_i x')=0,
 \qquad
 \limsup_{i\to\infty}\rho(s_i x,s_i x')>\delta_{y,\mathbf{s}}.
\]

\item
There exists a constant $\delta_*>0$, depending only on the
extension and on $\rho$, such that for every injective sequence
$\mathbf{s}=(s_i)_{i\geq1}\subseteq G$, there exists a measurable set
$
 Y_{\mathbf{s}}^*\subseteq Y_{\na},
$ with $
 \nu(Y_{\mathbf{s}}^*)>0,$
with the following property: for every $y\in Y_{\mathbf{s}}^*$, there is
a dense Mycielski set
$
 M_{y,\mathbf{s}}\subseteq\supp(\mu_y)
$
such that all distinct $x,x'\in M_{y,\mathbf{s}}$ satisfy
\[
 \liminf_{i\to\infty}\rho(s_i x,s_i x')=0,
 \qquad
 \limsup_{i\to\infty}\rho(s_i x,s_i x')>\delta_*.
\]
\end{enumerate}
\end{theorem}

We now return to Theorem~\ref{thm:sofic}.  By the Kerr--Li variational
principle, positive topological sofic entropy provides an invariant
measure of positive sofic measure entropy \cite{KerrLiVP}. Hayes
proved that every sofic p.m.p. action is relatively mixing over its
 sofic Pinsker factor \cite[Theorem~3.4 (i)]{Hayes}. Positive entropy ensures
that this factor extension is not an isomorphism, and therefore
Theorem~\ref{thm:abstract} \textup{(ii)} yields
Theorem~\ref{thm:sofic}. 

The relative-mixing criterion also applies beyond the sofic category. Seward proved that an essentially free p.m.p. action with completely
positive outer Rokhlin entropy relative to a factor is relatively mixing
\cite[Corollary~5.2 (1)]{SewardKoopman}.  Applying this result to the outer
Rokhlin Pinsker factor gives the following theorem.

\begin{theorem}\label{thm:rokhlin}
Let $G$ be a countably infinite group and let
$G\curvearrowright(X,\rho)$ be a continuous action on a compact metrizable
space.  Suppose that the action preserves a Borel probability measure $\mu$
such that the p.m.p. action $G\curvearrowright(X,\mu)$ is essentially free and the Rokhlin entropy
$
 h_G^{\mathrm{Rok}}(X,\mu)>0.
$
Then there exists $\delta>0$ such that, for every sequence
$\mathbf{s}=(s_i)_{i\geq1}$ of pairwise distinct elements of $G$, there is a Cantor set
$K_{\mathbf{s}}\subseteq X$ for which every distinct $x,x'\in K_{\mathbf{s}}$ satisfy
\[
 \liminf_{i\to\infty}\rho(s_i x,s_i x')=0,
 \qquad
 \limsup_{i\to\infty}\rho(s_i x,s_i x')>\delta.
\]
The constant $\delta$ is independent of $\mathbf{s}$.
\end{theorem}

The conclusion of Theorem~\ref{thm:sofic} does not extend from positive
sofic entropy to positive naive topological entropy. Indeed, the following
example gives a negative answer to Garc\'ia-Ramos and Li's
Question~5.28 \cite{GarciaRamosLiLocalEntropy}, already for a residually
finite group.

\begin{theorem}\label{thm:naive-counterexample}
Let
\[
 G=\langle a_1,a_2,\ldots\mid a_n^2=e\ (n\geq1)\rangle.
\]
Then $G$ is countably infinite and residually finite, and hence sofic.
There exists a compact metrizable hereditary subshift
$X\subseteq\{0,1\}^G$ such that the topological naive entropy
\[
 h_{\mathrm{nv}}^{\mathrm{top}}(X,G)>0
\]
and, for every compatible metric $\rho$ on $X$ and all $x,x'\in X$,
\[
 \lim_{n\to\infty}\rho(a_nx,a_nx')=0.
\]
In particular, there is no Li--Yorke pair along the injective sequence $(a_n)_{n\geq1}$.
\end{theorem}

\subsection*{Organization of the paper}
Section~\ref{sec:pre} collects the measure-theoretic preliminaries.
In Section~\ref{sec:relative-mixing-tools}, we develop the structural and
averaging tools for relatively mixing extensions that are needed in the
proof of Theorem~\ref{thm:abstract}.
Section~\ref{sec:proof-abstract} combines these ingredients with a
fiberwise Mycielski argument to prove Theorem~\ref{thm:abstract} and also gives a higher-order version, together with a brief proof sketch.
 Sections~\ref{sec:sofic} and~\ref{sec:rokhlin} apply the abstract
criterion to the sofic Pinsker factor and the outer Rokhlin
Pinsker factor, respectively, yielding Theorems~\ref{thm:sofic}
and~\ref{thm:rokhlin}. Section~\ref{sec:naive-counterexample} proves Theorem \ref{thm:naive-counterexample}.

\section{Preliminaries}\label{sec:pre}
Throughout this paper, $G$ denotes a countably infinite discrete group.
The notation $g\to\infty$ means that $g$ eventually lies outside every
finite subset of $G$. We write $\mathbb N$ for the set of natural numbers.
\subsection{Dynamical systems}

A \emph{$G$-system} is a compact metrizable space $X$ equipped with a
continuous action of $G$. Since $G$ is discrete, this is equivalent to an
action of $G$ on $X$ by homeomorphisms. We write $\Prob(X)$ for the space
of Borel probability measures on $X$ and $\M_G(X)$ for the set of
$G$-invariant Borel probability measures on $X$. When $G$ is amenable, $\M_G(X)$  is nonempty. For a general countable group, however,
$\M_G(X)$ may be empty.

A \emph{probability-measure-preserving $G$-system}, abbreviated as a
\emph{p.m.p.\ $G$-system}, is a standard probability space
$
(Z,m_Z)
$
equipped with a measurable action of $G$ such that
$
g_*m_Z=m_Z$
 for every $g\in G$.
In particular, if $(X,G)$ is a  $G$-system and
$\mu\in\M_G(X)$, then $(X,\mu,G)$ is a p.m.p.\ $G$-system.

Let
$
(Z,m_Z,G)
$ and $(W,m_W,G)
$
be p.m.p.\ $G$-systems. A measurable map
$
\pi:Z\to W
$
is called a \emph{measure-theoretic factor map} if
$
\pi_*m_Z=m_W
$
and, for every $g\in G$,
$
\pi(gz)=g\pi(z)
$
for $m_Z$-a.e. $z\in Z$. Since $G$ is countable, the
equivariance identities may be arranged to hold simultaneously for all
$g\in G$ on a single conull measurable subset of $Z$.

A measure-theoretic factor map
$
\pi:(Z,m_Z,G)\to(W,m_W,G)
$
is called a \emph{measure-theoretic isomorphism} if there exist invariant
conull Borel sets
$
Z_0\subseteq Z
$ and $W_0\subseteq W$
such that
$
\pi(Z_0)=W_0
$
and the restriction
$
\pi|_{Z_0}:Z_0\to W_0$
is a bimeasurable $G$-equivariant bijection.

Throughout the paper, measure-theoretic objects are understood modulo null
sets. In particular, if $A,B$ are Borel in a p.m.p.\ $G$-system
$(Z,m_Z,G)$, we write $A=B$ whenever
$
m_Z(A\mathbin{\triangle}B)=0.
$
Likewise, measurable functions and maps that agree almost everywhere are
identified, and measurable $\sigma$-algebras and factors are understood
modulo null sets. Conditional expectations and disintegrations are
understood up to their usual almost-everywhere uniqueness.
\subsection{Disintegration and relative products}
Let
$
\pi:(X,\mu,G)
\to
(Y,\nu,G)
$
be a measure-theoretic factor map. We write
$
\mu=\int_Y\mu_y\,d\nu(y)
$
for a disintegration of $\mu$ over $Y$. Thus, for every
$f\in L^1(X,\mu)$,
\[
\E^X_Y f(y)
:=
\int_X f\,d\mu_y
\]
defines a version of the conditional expectation of $f$ with respect to
$\pi$.

The \emph{relatively independent self-joining} of $\mu$ over $Y$ is
defined by
\[
\mu\times_Y\mu
:=
\int_Y\mu_y\otimes\mu_y\,d\nu(y).
\]
We denote this measure by
$
\lambda:=\mu\times_Y\mu.
$
The corresponding factor map
\[
\widetilde\pi:(X\times X,\lambda)\to(Y,\nu)
\]
is defined almost everywhere by
\[
\widetilde\pi(x,x')
=
\pi(x)
=
\pi(x').
\]
The group $G$ acts diagonally on $X\times X$ by
\[
g(x,x')=(gx,gx').
\]

For a p.m.p.\ $G$-system $(Z,m_Z,G)$, the \emph{Koopman representation} is
denoted by
\[
U_gf(z):=f(g^{-1}z),
\qquad
g\in G.
\]
We use the convention
\[
\langle \xi,\eta\rangle_{L^2(Z,m_Z)}
:=
\int_Z\xi\,\overline{\eta}\,dm_Z.
\]

\section{Relative mixing and injective-sequence averaging}
\label{sec:relative-mixing-tools}

This section collects the structural and averaging consequences of relative
mixing that will be used in the proof of Theorem~\ref{thm:abstract}. We prove
that relative mixing passes to the relatively independent square, establish
a relative Blum--Hanson type lemma along arbitrary injective sequences, and
derive a Dirac--nonatomic dichotomy for the conditional measures.

\begin{definition}\cite[Definition~3.3]{Hayes}\label{def:relmix}
The extension $
\pi:(X,\mu,G)
\to
(Y,\nu,G)
$ is \emph{relatively mixing} if, for every
$f,h\in L^\infty(X,\mu)$ satisfying $\E^X_Yf=\E^X_Yh=0$,
\[
 \bigl\|\E^X_Y(U_gf\,h)\bigr\|_{L^2(Y,\nu)}\longrightarrow0
 \qquad(g\to\infty).
\]
\end{definition}

Disintegrate $\mu$ over $\pi$ by
$
 \mu=\int_Y\mu_y\,d\nu(y).
$
We make the following standing version choice. Since $G$ is countable, after
modifying the factor map and the disintegration on null sets, there are
invariant conull Borel sets $X_0\subseteq X$ and
$Y_0\subseteq Y$ such that, for every $g\in G$,
\begin{equation}\label{eq:versions}
 \pi(gx)=g\pi(x)\quad(x\in X_0),
 \qquad
 \mu_{gy}=g_*\mu_y\quad(y\in Y_0),
\end{equation}
and, for every $y\in Y_0$, the measure $\mu_y$ is concentrated
on $X_0\cap\pi^{-1}(\{y\})$. 

\begin{remark}
\label{rem:set-formulation}
With the convention $gB=\{gx:x\in B\}$, Definition~\ref{def:relmix}
is equivalent to
\begin{equation}\label{eq:set-form}
     \int_Y
 \left|
 \mu_y(A\cap gB)-\mu_y(A)\mu_y(gB)
 \right|^2\,d\nu(y)
 \longrightarrow 0
 \qquad (g\to\infty)
\end{equation}
for all Borel sets $A,B\subseteq X$. Indeed,
\[
 U_g\one_B=\one_{gB},
 \qquad
 \E^X_Y(U_g\one_B)(y)
 =\mu_{g^{-1}y}(B)
 =\mu_y(gB),
\]
and hence the displayed integrand is the squared modulus of
\[
 \E^X_Y(U_g\one_B\,\one_A)
 -
 U_g^Y(\E^X_Y\one_B)\,\E^X_Y\one_A.
\]
The converse follows first for simple functions by linearity and then for
bounded functions by $L^2$-approximation and the contractivity of
conditional expectation. This is the set formulation used in
\cite[Section~5, immediately before Corollary~5.2]{SewardKoopman}.
\end{remark}
\subsection{Relative mixing of the relative square}
\label{subsec:relative-square}
Let
$
\pi:(X,\mu,G)
\to
(Y,\nu,G)
$ be a factor map, 
and denote the relatively independent self-joining of $\mu$ over $Y$ by
$\lambda:=\mu\times_Y\mu$.
\begin{lemma}\label{lem:relative-square}
If $\pi:(X,\mu,G)\to (Y,\nu,G)$ is relatively mixing, then
$\widetilde\pi:(X\times X,\lambda,G)\to(Y,\nu,G)$ is relatively mixing.
\end{lemma}

\begin{proof}
Write
\[
 \lambda_y:=\mu_y\otimes\mu_y,\qquad \nu\text{-a.e. } y\in Y.
\]
By Remark~\ref{rem:set-formulation}, it is enough to verify the
set-theoretic mixing condition \eqref{eq:set-form}. Consider first relative rectangles
\[
 A=A_1\times A_2,
 \qquad\text{and}\qquad
 B=B_1\times B_2,
\]
where $A_i,B_i\subseteq X$ are Borel. Then for $\nu$-a.e. $y\in Y$,
\[
 \lambda_y(A\cap gB)
 =
 \prod_{i=1}^2\mu_y(A_i\cap gB_i),
\qquad
 \lambda_y(A)\lambda_y(gB)
 =
 \prod_{i=1}^2
 \mu_y(A_i)\mu_y(gB_i).
\]
Since all the factors take values in $[0,1]$,  it follows from the relative mixing of $\pi$ that
\begin{align*}
&
 \left\|
 \lambda_y(A\cap gB)
 -
 \lambda_y(A)\lambda_y(gB)
 \right\|_{L^2(Y,\nu)}
\\
&\quad\leq
 \sum_{i=1}^2
 \left\|
 \mu_y(A_i\cap gB_i)
 -
 \mu_y(A_i)\mu_y(gB_i)
 \right\|_{L^2(Y,\nu)}
 \to0, \qquad\text{as }g\to\infty.
\end{align*}
The proof is finished, as the linear span of the indicators of such relative rectangles is dense in
$L^2(X\times X,\lambda)$. 
\end{proof}
\subsection{Relative Blum--Hanson averaging}
\label{subsec:relative-blum-hanson}
The following lemma is a relative, countable-group version of the
classical Blum--Hanson subsequence theorem
\cite{BlumHanson1960}; see also \cite{BerendBergelson1986} for an
abstract Hilbert-space treatment of mixing sequences.

\begin{lemma}\label{lem:averaging}
Let 
$\pi:(X,\mu,G)
   \to
   (Y,\nu,G)$
be a relatively mixing factor map. 
Let $(t_i)_{i\geq1}$ be a sequence of pairwise distinct elements of
$G$.  If $f\in L^\infty(X,\mu)$ satisfies
$
\E^X_Yf=0,
$
then
\begin{equation}\label{eq:averaging}
 \left\|
 \frac1N\sum_{i=1}^N U_{t_i}f
 \right\|_{L^2(X,\mu)}
 \longrightarrow 0
 \qquad\text{as }N\to\infty.
\end{equation}
\end{lemma}

\begin{proof}
For $g\in G$, define the matrix coefficient
\[
 c(g)
 :=
 \langle U_gf,f\rangle_{L^2(X,\mu)}
 =
 \int_X U_gf\,\overline f\,d\mu.
\]
Since
$
\E^X_Y\overline f
=
\overline{\E^X_Yf}
=
0,
$
relative mixing, applied to $f$ and $\overline f$, gives
\begin{equation}\label{eq:coefficient-decay}
 |c(g)|
 =
 \left|
 \int_Y \E^X_Y(U_gf\,\overline f)\,d\nu
 \right| \leq
 \bigl\|
 \E^X_Y(U_gf\,\overline f)
 \bigr\|_{L^2(Y,\nu)}\to0
 \qquad\text{as }g\to\infty.
\end{equation}

Fix $\varepsilon>0$.  By \eqref{eq:coefficient-decay}, there exists
a finite set $L\subseteq G$ such that
\[
 |c(g)|<\varepsilon
 \qquad\text{for every }g\notin L.
\]
Note that
\[
 \left\|
 \frac1N\sum_{i=1}^N U_{t_i}f
 \right\|_{L^2(X,\mu)}^2
=
 \frac1{N^2}
 \sum_{i,j=1}^N
 \langle U_{t_j^{-1}t_i}f,f\rangle_{L^2(X,\mu)}
=
 \frac1{N^2}
 \sum_{i,j=1}^N c(t_j^{-1}t_i).
\]
For each fixed $j\in\{1,\ldots,N\}$ and each $\ell\in L$, the
equation
$
 t_j^{-1}t_i=\ell
$
has at most one solution $i$, since the elements $t_i$ are pairwise
distinct.  Therefore
\[
 \#\left\{
 (i,j)\in\{1,\ldots,N\}^2:
 t_j^{-1}t_i\in L
 \right\}
 \leq |L|N.
\]
Since
$
 |c(g)|
 \leq
 \|U_gf\|_{L^2(X,\mu)}\|f\|_{L^2(X,\mu)}
 =
 \|f\|_{L^2(X,\mu)}^2,
$
we conclude that
 \begin{align*}
 \left\|
 \frac1N\sum_{i=1}^N U_{t_i}f
 \right\|_{L^2(X,\mu)}^2
 &\leq
 \frac1{N^2}
 \sum_{\substack{1\leq i,j\leq N\\
                  t_j^{-1}t_i\notin L}}
 |c(t_j^{-1}t_i)|
 +
 \frac1{N^2}
 \sum_{\substack{1\leq i,j\leq N\\
                  t_j^{-1}t_i\in L}}
 |c(t_j^{-1}t_i)|
 \\
 &\leq
 \varepsilon
 +
 \frac{|L|}{N}\|f\|_{L^2(X,\mu)}^2.
 \end{align*}
Taking $\limsup_{N\to\infty}$ gives
\[
 \limsup_{N\to\infty}
 \left\|
 \frac1N\sum_{i=1}^N U_{t_i}f
 \right\|_{L^2(X,\mu)}^2
 \leq\varepsilon.
\]
Since $\varepsilon>0$ was arbitrary, \eqref{eq:averaging} follows.
\end{proof}
For the later application to the relative square, we need the same
subsequence of averaging lengths to work simultaneously for a
countable family of centered functions.  This follows from the
preceding $L^2$-convergence by a standard diagonal argument.
\begin{lemma}
\label{lem:simultaneous}
Assume the hypotheses of Lemma~\ref{lem:averaging}.  Let
$
f_1,f_2,\ldots\in L^\infty(X,\mu)
$
satisfy
$
\E^X_Yf_r=0$, for every $r\geq1$.
Then there exists a strictly increasing sequence of positive integers $\{N_k\}_{k\in\mb N}$
such that, simultaneously for every $r\geq1$,
\begin{equation}\label{eq:pointwise-subsequence}
 \lim_{k\to\infty}\frac1{N_k}\sum_{i=1}^{N_k}U_{t_i}f_r(x)
 =0\qquad\text{for $\mu$-a.e. $x\in X$.}
\end{equation}
\end{lemma}

\begin{proof}
For $N,r\geq1$, write
\[
 A_{N,r}
 :=
 \frac1N\sum_{i=1}^N U_{t_i}f_r.
\]
By Lemma~\ref{lem:averaging}, for every fixed $r$,
\[
 \|A_{N,r}\|_{L^2(X,\mu)}
 \longrightarrow0
 \qquad\text{as }N\to\infty.
\]

We now choose $N_k$ inductively. Put $N_0:=0$. Having chosen $N_{k-1}$, the
$L^2$-convergence above, applied to the finite family
$f_1,\ldots,f_k$, allows us to choose
$
 N_k>N_{k-1}
$
so large that
\begin{equation}\label{eq:diagonal-L2-bound}
 \|A_{N_k,r}\|_{L^2(X,\mu)}^2<2^{-k}
 \qquad(1\leq r\leq k).
\end{equation}

Fix $r\geq1$.  By \eqref{eq:diagonal-L2-bound},
\[
 \sum_{k=r}^{\infty}
 \|A_{N_k,r}\|_{L^2(X,\mu)}^2
 \leq
 \sum_{k=r}^{\infty}2^{-k}
 <\infty.
\]
Since the summands are nonnegative, Tonelli's theorem gives
\[
 \begin{aligned}
 \int_X
 \sum_{k=r}^{\infty}
 |A_{N_k,r}(x)|^2\,d\mu(x)
 &=
 \sum_{k=r}^{\infty}
 \int_X|A_{N_k,r}(x)|^2\,d\mu(x)\\
 &=
 \sum_{k=r}^{\infty}
 \|A_{N_k,r}\|_{L^2(X,\mu)}^2
 <\infty.
 \end{aligned}
\]
Consequently,
\[
 \sum_{k=r}^{\infty}|A_{N_k,r}(x)|^2<\infty \qquad\text{for $\mu$-a.e. $x\in X$.}
\]
  In particular,
\[
\lim_{k\to\infty} A_{N_k,r}(x)=0 \qquad\text{for $\mu$-a.e. $x\in X$.}
\]
For each $r$, let $X_r\subseteq X$ be a conull set on which this
convergence holds, and set
$
 X_0:=\bigcap_{r=1}^{\infty}X_r.
$
As the intersection is countable,
$
 \mu(X_0)=1.
$
For every $x\in X_0$, the convergence holds simultaneously for all
$r\geq1$.
\end{proof}

\subsection{The atomic structure of the conditional measures}
\label{subsec:conditional-measures}

Relative (weak) mixing also imposes a rigid dichotomy on the conditional
measures of the extension.  Namely, an atomic conditional measure must
in fact be a Dirac measure.  This observation identifies precisely when
a relatively mixing extension has nontrivial nonatomic fibers.

\begin{lemma}\label{lem:dichotomy}
Let
$
\pi:(X,\mu,G)\to(Y,\nu,G)
$
be a relatively  mixing factor map\footnote{In fact, the proof only requires
the extension $\pi:(X,\mu,G)\to(Y,\nu,G)$ to be relatively weakly mixing;
see \cite[Chapter~9, Section~5]{GlasnerJoinings} for definition of relatively weakly mixing.}. Then, for $\nu$-a.e. $y\in Y$,
the conditional measure $\mu_y$ is either a Dirac measure or nonatomic.
Moreover, 
$
\nu(Y_{\na})=0$
 if and only if
$\pi
$
is an isomorphism, where $Y_{\na}:=\{y\in Y:\mu_y\text{ is nonatomic}\}$.
\end{lemma}

\begin{proof}
We use the standard characterization of relative weak mixing by relative
ergodicity of the relatively independent self-joining; see, for example,
\cite[Chapter~9, Section~5]{GlasnerJoinings}. Since relative mixing implies
relative weak mixing, the extension
$
(X\times X,\mu\times_Y\mu,G)\to(Y,\nu,G)
$
is relatively ergodic.

Put
\[
d(y):=(\mu_y\otimes\mu_y)(\Delta_X)
     =\E_Y^{X\times X}(\one_{\Delta_X})(y),
\]
where $
 \Delta_X:=\{(x,x):x\in X\}\subseteq X\times X.
$
Since $\Delta_X$ is invariant under the diagonal action, relative
ergodicity implies
\[
\one_{\Delta_X}=d\circ\widetilde\pi
\qquad
(\mu\times_Y\mu)\text{-a.e}.
\]
Since $\one_{\Delta_X}$ takes only the values $0$ and $1$, it follows that
\[
d(y)\in\{0,1\}
\qquad
\text{for }\nu\text{-a.e. }y.
\]
This
implies the asserted dichotomy.

Choose an equivariant version of the disintegration such that
\[
\mu_{gy}=g_*\mu_y,
\qquad
g\in G
\]
on a fixed invariant conull subset of $Y$. The property of being
nonatomic is preserved by pushforward under a bijection. Therefore
$Y_{\na}$ is $G$-invariant.

Suppose that
$
\nu(Y_{\na})=0.
$
Then $\mu_y$ is a Dirac measure for $\nu$-a.e. $y$. Since the map
$
y\mapsto\mu_y
$
is measurable and the Dirac embedding
\[
X\longrightarrow\Prob(X),
\qquad
x\longmapsto\delta_x,
\]
is a Borel isomorphism onto its image, there exists a measurable map
$
\xi:Y\to X
$
such that
\[
\mu_y=\delta_{\xi(y)}\qquad\text{for $\nu$-a.e. $y$.}
\]
The support property of the disintegration gives
\[
\pi(\xi(y))=y
\qquad
\text{for }\nu\text{-a.e. }y.
\]
Moreover,
\[
\mu\bigl(\{x\in X:\xi(\pi(x))=x\}\bigr)
=
\int_Y
\delta_{\xi(y)}
\bigl(\{x\in X:\xi(\pi(x))=x\}\bigr)
\,d\nu(y)
=
1.
\]
Thus
\[
\pi\circ\xi=\id_Y
\pmod\nu,
\qquad
\xi\circ\pi=\id_X
\pmod\mu.
\]
Finally,
\[
\delta_{\xi(gy)}
=
\mu_{gy}
=
g_*\mu_y
=
g_*\delta_{\xi(y)}
=
\delta_{g\xi(y)}
\]
for every $g\in G$ and $\nu$-a.e. $y$. Hence
 $\xi$ is an equivariant inverse of $\pi$.

Conversely, if $\pi$ is an isomorphism, then its
conditional measures are Dirac almost everywhere. Therefore
$
\nu(Y_{\na})=0.
$
\end{proof}
\section{Proof of Theorem~\ref{thm:abstract} and  higher-order extensions}
\label{sec:proof-abstract}
Throughout this section, let
$(X,G)$ be a $G$-system equipped with a compatible metric $\rho$, and
let
$
\pi:(X,\mu,G)\to(Y,\nu,G)
$
be a relatively  mixing factor map, which is not an isomorphism. Let
$
\lambda
:=
\mu\times_Y\mu
$
be the relatively independent self-joining, and 
$
\widetilde\pi:(X\times X,\lambda)\to(Y,\nu)
$
be the associated factor map. Since the extension is not an isomorphism, Lemma~\ref{lem:dichotomy} gives $\nu(Y_{\mr{na}})>0.$

The proof of Theorem~\ref{thm:abstract} has two main steps. We first establish
fiberwise close and separated visits along the prescribed injective sequence,
and then apply Mycielski's theorem to the conditional supports. In the final
subsection, we formulate a higher-order version and provide a brief proof
sketch.
\subsection{Fiberwise close and separated visits}
\label{subsec:fiberwise-visits}
We first show that, along the prescribed injective sequence, almost every
pair in almost every nonatomic conditional fiber visits arbitrarily small
neighborhoods of the diagonal and also visits a fixed complement of the
diagonal infinitely often. 

 For $m\geq1$, put
\[
A_m
:=
\{(x,x')\in X\times X:\rho(x,x')<1/m\}
\]
and
\[
p_m(y)
:=
(\mu_y\otimes\mu_y)(A_m).
\]
Choose a finite Borel partition
\[
X=C_{m,1}\sqcup\cdots\sqcup C_{m,r_m}
\]
whose atoms have diameter less than $1/m$. Then by Cauchy--Schwarz inequality,
\begin{equation}\label{eq:uniform-close-mass}
p_m(y)
\geq
\sum_{j=1}^{r_m}\mu_y(C_{m,j})^2
\geq
\frac{1}{r_m}
=:c_m>0\qquad\text{for every $y\in Y$}.
\end{equation}

For $n\geq1$, put
\[
B_n
:=
\{(x,x')\in X\times X:\rho(x,x')>1/n\}
\]
and
\[
q_n(y)
:=
(\mu_y\otimes\mu_y)(B_n).
\]
For every $y\in Y_{\na}$, nonatomicity of $\mu_y$ gives
\[
q_n(y)
\uparrow
(\mu_y\otimes\mu_y)
\bigl((X\times X)\setminus\Delta_X\bigr)
=
1.
\]
Define
\[
V_n
:=
Y_{\na}\cap\{y\in Y:q_n(y)>1/2\}.
\]
Then $(V_n)_{n\geq1}$ is increasing and
\begin{equation}\label{eq:Vn-union}
\bigcup_{n\geq1}V_n
=
Y_{\na}\pmod \nu.
\end{equation}

We now fix a sequence $(s_i)_{i\geq1}$ of pairwise distinct elements of $G$.
On the relative square, define
\[
F_m
:=
\one_{A_m}-p_m\circ\widetilde\pi,
\qquad
G_n
:=
\one_{B_n}-q_n\circ\widetilde\pi,\qquad\text{for } m,n\in\mb N.
\]
These are bounded functions with zero conditional expectation over $Y$.

Apply Lemma~\ref{lem:simultaneous} to the injective sequence
$
t_i:=s_i^{-1}
$
and to the countable family $\{F_m\}_{m\in\mb N}$ and $\{G_n\}_{n\in\mb N}.$
There exist integers $N_k\to\infty$ and a measurable set
$
\Omega\subseteq X\times X,$ with $\lambda(\Omega)=1$
such that, for every $z\in\Omega$, writing
$
y:=\widetilde\pi(z),
$
one has
\begin{align}
\lim_{k\to\infty}\frac{1}{N_k}\sum_{i=1}^{N_k}
\bigl(
\one_{A_m}(s_i z)-p_m(s_i y)
\bigr)
&=0
&&\text{for every }m,
\label{eq:close-average}\\
\lim_{k\to\infty}\frac{1}{N_k}\sum_{i=1}^{N_k}
\bigl(
\one_{B_n}(s_i z)-q_n(s_i y)
\bigr)
&=0
&&\text{for every }n.
\label{eq:far-average}
\end{align}
Here we have also intersected $\Omega$ with the fixed invariant conull set
on which
$
\widetilde\pi(s_i z)=s_i\widetilde\pi(z)
$
holds for every $i$.

By \eqref{eq:uniform-close-mass} and \eqref{eq:close-average}, for every
$m\geq1$ and every $z\in\Omega$,
\[
\liminf_{k\to\infty}
\frac{1}{N_k}\sum_{i=1}^{N_k}\one_{A_m}(s_i z)
\geq
c_m>0.
\]
Hence, for every $m\geq1$ and for
$\lambda$-a.e. $(x,x')$,
\begin{equation}\label{eq:close-infinitely-often}
\#\bigl\{i\geq 1:
\rho(s_i x,s_i x')<1/m\bigr\}
=
\infty.
\end{equation}

We next obtain separated visits. For $n,k\geq1$, define
\[
r_{k,n}(y)
:=
\frac{1}{N_k}\sum_{i=1}^{N_k}\one_{V_n}(s_i y)
\]
and
\[
D_n
:=
\left\{
y\in Y:
\limsup_{k\to\infty}r_{k,n}(y)>0
\right\}.
\]
Since $\nu$ is $G$-invariant,
$
\int_Y r_{k,n}\,d\nu
=
\nu(V_n)
$
for every $k\in\mb N$. The reverse Fatou's lemma   gives
$$
\int_Y\limsup_{k\to\infty}r_{k,n}\,d\nu
\geq
\nu(V_n).
$$
Since
$
0\leq\limsup_{k\to\infty}r_{k,n}\leq\one_{D_n},
$
we obtain
\begin{equation}\label{eq:Dn-lower}
\nu(D_n)\geq\nu(V_n).
\end{equation}

The sets $D_n$ are increasing as $n\to\infty$. Moreover, since $Y_{\na}$ is
$G$-invariant and $V_n\subseteq Y_{\na}$, one has
$
D_n\subseteq Y_{\na}
$.
Together with \eqref{eq:Vn-union} and \eqref{eq:Dn-lower}, this implies
\begin{equation}\label{eq:Dn-union}
\bigcup_{n\geq1}D_n
=
Y_{\na}\pmod \nu.
\end{equation}

Disintegrating the identity $\lambda(\Omega)=1$, choose a conull measurable
set $Y_0\subseteq Y$ such that
\[
(\mu_y\otimes\mu_y)(\Omega)=1
\qquad
\text{for every }y\in Y_0.
\]
Set
\begin{equation*}
    Y_{(s_i)}
:=
Y_0\cap Y_{\na}\cap\bigcup_{n\geq1}D_n.
\end{equation*}
By \eqref{eq:Dn-union}, the set $Y_{(s_i)}=Y_{\na}\pmod \nu$.

Fix $y\in Y_{(s_i)}$, and choose an integer $n(y)\geq1$ such that
$
y\in D_{n(y)}.
$
Since
$
q_{n(y)}
\geq
\frac{1}{2}\one_{V_{n(y)}},
$
we have
\[
\limsup_{k\to\infty}
\frac{1}{N_k}\sum_{i=1}^{N_k}q_{n(y)}(s_i y)
\geq
\frac{1}{2}
\limsup_{k\to\infty}r_{k,n(y)}(y)
>
0.
\]
Combining this with \eqref{eq:far-average}, we obtain, for
$(\mu_y\otimes\mu_y)$-a.e. $z=(x,x')$,
\[
\limsup_{k\to\infty}
\frac{1}{N_k}\sum_{i=1}^{N_k}
\one_{B_{n(y)}}(s_i z)
>
0.
\]
Therefore, for $\mu_y\otimes\mu_y$-a.e.\ $(x,x')\in X\times X$,
\begin{equation}\label{eq:far-infinitely-often}
\#\left\{
i\geq 1:
\rho(s_i x,s_i x')>\frac{1}{n(y)}
\right\}
=
\infty.
\end{equation}

\subsection{The Mycielski construction}
\label{subsec:mycielski}
We now convert the preceding full conditional-measure relations into
topologically large scrambled sets and complete the proof of
Theorem~\ref{thm:abstract}.

Fix $y\in Y_{(s_i)}$, and put
\[
X_y:=\supp(\mu_y)
\qquad\text{and}\qquad
\delta_y:=\frac{1}{2n(y)}.
\]
Define a relation $R_y\subseteq X_y\times X_y$ by
\begin{align*}
R_y
:=&
\bigcap_{m=1}^{\infty}
\bigcap_{L=1}^{\infty}
\bigcup_{i\geq L}
\left\{
(x,x')\in X_y^2:
\rho(s_i x,s_i x')<\frac{1}{m}
\right\}
\\
&\cap
\bigcap_{L=1}^{\infty}
\bigcup_{i\geq L}
\left\{
(x,x')\in X_y^2:
\rho(s_i x,s_i x')>\frac{1}{n(y)}
\right\}.
\end{align*}
By continuity of the action, $R_y$ is a  $G_\delta$ subset of
$X_y^2$. Equations \eqref{eq:close-infinitely-often} and
\eqref{eq:far-infinitely-often} give
\[
(\mu_y\otimes\mu_y)(R_y)=1.
\]
Since $\mu_y$ has full support on $X_y$, every nonempty relatively open
subset of $X_y^2$ has positive $\mu_y\otimes\mu_y$ measure. Hence $R_y$ is
dense in $X_y^2$.

The compact metrizable space $X_y$ is perfect. Indeed, if $x\in X_y$ were
isolated in $X_y$, then there would be an open set $U\subseteq X$ such that
$
U\cap X_y=\{x\}.
$
Since $x\in\supp(\mu_y)$, one would have
$
\mu_y(\{x\})=\mu_y(U)>0,
$
contradicting the nonatomicity of $\mu_y$.

By Mycielski's theorem
\cite[Theorem~1]{Mycielski}, there exists a dense Mycielski set
$
M_y\subseteq X_y
$
such that
$
(M_y\times M_y)\setminus\Delta_X
\subseteq
R_y.
$
Thus every distinct $x,x'\in M_y$ satisfies
\[
\liminf_{i\to\infty}\rho(s_i x,s_i x')=0,
\qquad
\limsup_{i\to\infty}\rho(s_i x,s_i x')
\geq
\frac{1}{n(y)}
>
\delta_y.
\]
This proves Theorem~\ref{thm:abstract}~\textup{(i)}.

\medskip
We now prove Theorem \ref{thm:abstract}~(ii). Since $\nu(Y_{\na})>0$ and $V_n\uparrow Y_{\na},$ choose, before any injective sequence is specified,
an integer $n_*\geq1$ such that
$
\nu(V_{n_*})>0,
$
and set
$
\delta_*:=\frac{1}{2n_*}.
$
Notice that $V_n$, and hence the choice of $n_*$, depends only on the
extension and on the metric $\rho$, and not on any prescribed sequence.

Now let $(s_i)_{i\geq1}$ be an arbitrary sequence of pairwise distinct
elements of $G$. Apply the construction of
Subsection~\ref{subsec:fiberwise-visits} to this sequence. It gives a
conull set $Y_0\subseteq Y$ and measurable sets $D_n\subseteq Y$ satisfying
\[
\nu(D_n)\geq\nu(V_n)
\qquad
\text{for every }n\geq1.
\]
In particular,
\[
\nu(D_{n_*})
\geq
\nu(V_{n_*})
>
0.
\]
Therefore,
\[
\nu\bigl(Y_0\cap D_{n_*}\cap Y_{\na}\bigr)>0.
\]

For any fixed $y\in  Y_0\cap D_{n_*}\cap Y_{\mathrm{na}}$, the preceding argument applies with
$n(y)=n_*$.
Repeating
the preceding Mycielski construction gives a dense Mycielski set
$
M_y\subseteq\supp(\mu_y)
$
such that every distinct $x,x'\in M_y$ satisfy
\[
\liminf_{i\to\infty}\rho(s_i x,s_i x')=0
\qquad\text{and}\qquad
\limsup_{i\to\infty}\rho(s_i x,s_i x')
\geq
\frac{1}{n_*}
>
\delta_*.
\]
Since $n_*$ was chosen before $(s_i)$ was fixed, this constant is
independent of the prescribed sequence. This proves
Theorem~\ref{thm:abstract} \textup{(ii)}.
\subsection{Higher-order scrambled sets}
\label{subsec:higher-order}
We conclude this section by recording a higher-order consequence of the
preceding argument.

\begin{proposition}
\label{prop:higher-order-scrambled}
Fix an integer $r\geq 2$. Then there exists a constant
$
\delta_r>0,
$
depending only on the extension
$
\pi:(X,\mu,G)\to(Y,\nu,G),
$
the compatible metric $\rho$, and $r$, such that the following holds.
For every injective sequence
$\mathbf{s}=(s_i)_{i\geq 1}$
in $G$, there exists a Cantor set
$
K_{r,\mathbf{s}}\subseteq X
$
such that, for every $r$-tuple of pairwise distinct points
$
x_1,\ldots,x_r\in K_{r,\mathbf{s}},
$
one has
\[
\liminf_{i\to\infty}
\max_{1\leq a<b\leq r}
\rho(s_i x_a,s_i x_b)
=
0
\quad\text{and}\quad
\limsup_{i\to\infty}
\min_{1\leq a<b\leq r}
\rho(s_i x_a,s_i x_b)
>
\delta_r.
\]
\end{proposition}

\begin{proof}[Proof sketch]
We indicate the modifications to the proof of
Theorem~\ref{thm:abstract}. Consider the $r$-fold
relatively independent joining
$
\lambda^{(r)}
:=
\int_Y \mu_y^{\otimes r}\,d\nu(y)
$
over $Y$, equipped with the diagonal $G$-action. The corresponding
extension
$
(X^r,\lambda^{(r)},G)\to(Y,\nu,G)
$
is relatively mixing. Indeed, for relative rectangles, this follows
from relative mixing of $X\to Y$ and the elementary inequality
\[
\left|
\prod_{j=1}^r a_j-\prod_{j=1}^r b_j
\right|
\leq
\sum_{j=1}^r|a_j-b_j|,
\qquad
a_j,b_j\in[0,1],
\]
and the general case follows by the same approximation argument as in
Lemma~\ref{lem:relative-square}.

For $m,n\geq 1$, put
\[
A_{r,m}
:=
\left\{
(x_1,\ldots,x_r)\in X^r:
\max_{1\leq a<b\leq r}\rho(x_a,x_b)<\frac1m
\right\}
\]
and
\[
B_{r,n}
:=
\left\{
(x_1,\ldots,x_r)\in X^r:
\min_{1\leq a<b\leq r}\rho(x_a,x_b)>\frac1n
\right\}.
\]
If $
X=C_{m,1}\sqcup\cdots\sqcup C_{m,\ell_m}
$
is a finite Borel partition whose atoms have diameter less than $1/m$,
then, Hölder's inequality gives that, for every $y\in Y$, 
\[
\mu_y^{\otimes r}(A_{r,m})
\geq
\sum_{j=1}^{\ell_m}\mu_y(C_{m,j})^r
\geq
\ell_m^{\,1-r}>0.
\]
On the other hand, if $y\in Y_{\mathrm{na}}$, then
\[
\mu_y^{\otimes r}(B_{r,n})\nearrow 1
\qquad
(n\to\infty),
\]
as $\mu_y$ is nonatomic.

Define
\[
V_{r,n}
:=
Y_{\mathrm{na}}
\cap
\left\{
y\in Y:
\mu_y^{\otimes r}(B_{r,n})>\frac12
\right\}.
\]
Then
\[
V_{r,n}\nearrow Y_{\mathrm{na}}
\pmod{\nu}.
\]
Since $\nu(Y_{\mathrm{na}})>0$, we may choose $n_r\geq 1$,
before the sequence $\mathbf{s}$ is specified, such that
$
\nu(V_{r,n_r})>0,
$
and set
$
\delta_r:=\frac1{2n_r}.
$

Now fix an injective sequence $\mathbf{s}=(s_i)_{i\geq 1}$.
Applying Lemma~\ref{lem:simultaneous} on
$(X^r,\lambda^{(r)})$, and repeating the averaging and Fatou
arguments from Subsection~\ref{subsec:fiberwise-visits}, we obtain a
point $y\in Y_{\mathrm{na}}$ such that, for
$\mu_y^{\otimes r}$-a.e.
$(x_1,\ldots,x_r)\in X^r$,
\[
\max_{1\leq a<b\leq r}
\rho(s_i x_a,s_i x_b)<\frac1m
\]
holds for infinitely many $i$, for every $m\geq 1$, and
\[
\min_{1\leq a<b\leq r}
\rho(s_i x_a,s_i x_b)>\frac1{n_r}
\]
also holds for infinitely many $i$.

Put
\[
X_y:=\operatorname{supp}(\mu_y).
\]
The set of $r$-tuples in $X_y^r$ satisfying the preceding two
infinitude conditions is a dense $G_\delta$ subset of $X_y^r$.
Since $\mu_y$ is nonatomic, $X_y$ is perfect. The higher-order
form of Mycielski's theorem \cite[Theorem~1]{Mycielski} therefore gives a dense Mycielski set
$
M_{r,y}\subseteq X_y
$
such that every $r$-tuple of pairwise distinct points of $M_{r,y}$
satisfies
\[
\liminf_{i\to\infty}
\max_{1\leq a<b\leq r}
\rho(s_i x_a,s_i x_b)
=
0
\]
and
\[
\limsup_{i\to\infty}
\min_{1\leq a<b\leq r}
\rho(s_i x_a,s_i x_b)
\geq
\frac1{n_r}
>
\delta_r.\qedhere
\]
\end{proof}

\section{Applications}\label{sec:applications}
The abstract criterion reduces prescribed-sequence Li--Yorke chaos to the
existence of a nontrivial relatively mixing factor extension. We now apply
this criterion to two canonical entropy factors. The sofic Pinsker
factor yields the positive-sofic-entropy consequence, while the outer
Rokhlin Pinsker factor gives an application to essentially free actions of
arbitrary countably infinite discrete groups.

\subsection{Positive sofic entropy}\label{sec:sofic}
We briefly recall the relevant notation. For $d\in\mathbb N$, let
$\operatorname{Sym}(d)$ denote the symmetric group of
$\{1,\ldots,d\}$, equipped with the normalized Hamming metric
\[
d_{\mathrm{Hamm}}(\tau,\omega)
:=
\frac{1}{d}
\bigl|
\{v\in\{1,\ldots,d\}:\tau(v)\neq\omega(v)\}
\bigr|.
\]
A sequence of maps
\[
\Sigma
=
(\sigma_i:G\to\operatorname{Sym}(d_i))_{i\geq1},
\qquad
d_i\to\infty,
\]
is called a \emph{sofic approximation} to $G$ if
\[
\lim_{i\to\infty}
d_{\mathrm{Hamm}}
\bigl(
\sigma_i(st),
\sigma_i(s)\sigma_i(t)
\bigr)
=
0
\]
for every $s,t\in G$, and
\[
\lim_{i\to\infty}
d_{\mathrm{Hamm}}
\bigl(
\sigma_i(s),
\sigma_i(t)
\bigr)
=
1
\]
for every distinct $s,t\in G$. A countable group is called
\emph{sofic} if it admits a sofic approximation.

Throughout this subsection, we fix a sofic approximation $\Sigma$ of $G$
and use the definitions of topological and measure sofic entropy from
\cite{KerrLiVP}, denoted respectively by
\[
h_\Sigma^{\mathrm{top}}(X,G)
\qquad\text{and}\qquad
h_{\Sigma,\mu}(X,G).
\]

\begin{proof}[Proof of Theorem~\ref{thm:sofic}]
By the variational principle for sofic entropy
\cite[Theorem~6.1]{KerrLiVP},
\[
h^{\mathrm{top}}_\Sigma(X,G)
=
\sup_{\mu\in\M_G(X)}
h_{\Sigma,\mu}(X,G).
\]
Hence there exists $\mu\in\M_G(X)$ such that
\[
h_{\Sigma,\mu}(X,G)>0.
\]

Let
$
\pi_\Sigma:
(X,\mu,G)
\to
(Y_\Sigma,\nu_\Sigma,G)
$
be the Pinsker factor associated with $\Sigma$; see
\cite[Definition~3.1]{Hayes}. By definition, this is the largest factor
having zero sofic measure entropy. In particular, $\pi_\Sigma$ cannot be an
isomorphism.

By \cite[Theorem~3.4(i)]{Hayes}, the extension
$\pi_\Sigma$
is relatively mixing. Therefore
Theorem~\ref{thm:abstract}~\textup{(ii)} applies and completes the proof.
\end{proof}

\subsection{Positive Rokhlin entropy}
\label{sec:rokhlin}

We use Seward's definitions and conventions for Rokhlin entropy, outer
Rokhlin entropy, and the outer Rokhlin Pinsker factor; see
\cite{SewardKoopman}. We write
\[
h_G^{\mathrm{Rok}}(X,\mu)
\]
for the Rokhlin entropy of the p.m.p.\ $G$-system
$(X,\mu,G)$.

\begin{proof}[Proof of Theorem~\ref{thm:rokhlin}]
Let
$
\pi_+:
(X,\mu,G)
\to
(Y_+,\nu_+,G)
$
be the outer Rokhlin Pinsker factor. Thus $Y_+$ is the smallest factor
relative to which the extension has completely positive outer Rokhlin
entropy, abbreviated as relative $\mathrm{CPE}^+$.

Since
$
h_G^{\mathrm{Rok}}(X,\mu)>0,
$
the outer Rokhlin Pinsker factor is proper; equivalently,
$\pi_+$ is not an isomorphism. Moreover,
$(X,\mu,G)$ is $\mathrm{CPE}^+$ relative to $Y_+$; see the discussion
preceding \cite[Corollary~5.1]{SewardKoopman}.

Since the p.m.p. action $G\curvearrowright(X,\mu)$ is essentially free, 
\cite[Corollary~5.2 (1)]{SewardKoopman} implies that
$\pi_+$
is mixing relative to $Y_+$ in Seward's set-theoretic formulation. By
Remark~\ref{rem:set-formulation}, this is equivalent to relative mixing in
the sense of Definition~\ref{def:relmix}.
The proof is completed by applying
Theorem~\ref{thm:abstract}~\textup{(ii)}.
\end{proof}

\section{A counterexample for positive naive entropy}
\label{sec:naive-counterexample}
In this section, we prove Theorem~\ref{thm:naive-counterexample}.
The construction is motivated in part by the example of Li and Rong,
who showed that positive naive topological entropy need not give rise
to non-diagonal orbit IE-pairs
\cite[Proposition~5.2]{LiRongNaive}. Our example exhibits a different
obstruction, with all pairs becoming asymptotic along a common injective
sequence. In fact, it is not difficult to verify that our example admits
a non-diagonal orbit IE-pair. Thus the existence of such a pair still
does not imply Li--Yorke chaos along every prescribed infinite subset
of the acting group. 

The proof of positive naive
entropy can also be naturally viewed within their independence-density
framework \cite[Theorem~2.5]{LiRongNaive}. In our construction, an
iterative selection argument retains a uniformly positive proportion
of coordinates inside every finite subset of the group, while heredity
allows all binary patterns to be realized on those coordinates.
At the same time, constraints on reduced words force all pairs to
become asymptotic along a common injective sequence.

For a finite open cover $\mathcal V$ of a compact $G$-space, recall that
\[
 h_{\mathrm{nv}}(G,\mathcal V)
 :=\inf_{\emptyset\neq F\Subset G}
       \frac{1}{|F|}\log N(\mathcal V^F),
 \qquad
 \mathcal V^F:=\bigvee_{t\in F}t^{-1}\mathcal V,
\]
where $N(\mathcal W)$ is the minimum size of a subcover of $\mathcal W$
and $F\Subset G$ means that $F\subset G$ is finite. The \emph{naive topological entropy} \cite{Burton2017}
is $$h_{\mathrm{nv}}^{\mathrm{top}}(X,G)
=\sup_{\mathcal V}h_{\mathrm{nv}}(G,\mathcal V).$$

Let
\begin{equation}\label{eq:nv-group}
 G:=*_{n\in\mathbb N} C_2
   =\langle a_1,a_2,\ldots\mid a_n^2=e,\ n\in\mathbb N\rangle,
\end{equation}
the free product of countably many copies of the cyclic group $C_2$.
 The group $G$ is residually finite. Indeed, every nontrivial
element involves only finitely many free factors, whose finite free
product is residually finite by Gruenberg's theorem
\cite{Gruenberg1957}, and the remaining free factors may be mapped
trivially. Since every countable residually finite group is sofic, $G$ is sofic.
We write $|g|$ for the reduced-word length of $g$, with $|e|=0$.

For $k\geq0$, denote
\[
 M_k:=2^{k+2}.
\]
Define
\[
 \mathcal S_k
 :=
 \{\emptyset\}
 \cup
 \bigl\{
 I\subseteq\mathbb N:
 I\text{ is finite and } |I|\leq M_k\min I
 \bigr\}.
\]
We then denote
\begin{equation}\label{eq:nv-family}
 \mathcal E_2:=\{E\subseteq\mathbb N:|E|\leq2\},
 \qquad
 \mathcal K_k
 :=
 \{B\cup E:B\in\mathcal S_k,\ E\in\mathcal E_2\}.
\end{equation}
We identify subsets of $\mathbb N$ with their indicator functions in
$2^{\mathbb N}=\{0,1\}^{\mathbb N}$, endowed with the product topology.
\begin{proposition}\label{prop:nv-families}
For every $k\geq0$, the family $\mathcal K_k$ is
compact and hereditary\footnote{A family $\mathcal F\subseteq 2^{\mathbb N}$
is called \emph{hereditary} if $A\in\mathcal F$ and $B\subseteq A$ imply
$B\in\mathcal F$.}
\end{proposition}

\begin{proof}
The family $\mathcal S_k$ is clearly hereditary. It is also closed. Indeed,
if $I\notin\mathcal S_k$ and $r=\min I$, then $I$ contains a finite set
$L$ with $\min L=r$ and $|L|=M_kr+1$, and every set containing $L$
lies outside $\mathcal S_k$. Hence $\mathcal S_k$ is compact.
On the other hand, it is clear that the family $\mathcal E_2$ is compact and hereditary. 

Since the union map
\[
 2^{\mathbb N}\times2^{\mathbb N}\to2^{\mathbb N},
 \qquad (B,E)\mapsto B\cup E,
\]
is continuous, $\mathcal K_k$ is compact. Its heredity follows from that
of $\mathcal S_k$ and $\mathcal E_2$. 
\end{proof}

For $A\subseteq G$ and $h,g\in G$, denote
\[
 I_{h,g}(A):=\{n\in\mathbb N:ha_ng\in A\},
\]
and define
\begin{equation}\label{eq:nv-subshift}
 X:=\bigl\{\one_A\in\{0,1\}^G:
       I_{h,g}(A)\in\mathcal K_{|g|}
       \text{ for every }h,g\in G\bigr\}.
\end{equation}
We equip $\{0,1\}^G$ with the left shift action
\[
 (tx)(q):=x(t^{-1}q),
 \qquad t,q\in G.
\]

\begin{proposition}\label{prop:nv-subshift}
The space $X$ is a nonempty compact metrizable hereditary subshift of
$\{0,1\}^G$.
\end{proposition}

\begin{proof}
For fixed $h,g\in G$, define
\[
 \Phi_{h,g}:\{0,1\}^G\to 2^{\mathbb N},
 \qquad
 \Phi_{h,g}(\one_A):=\one_{I_{h,g}(A)}.
\]
This map is continuous, since each of its coordinates is a coordinate
projection on $\{0,1\}^G$. Hence, by the closedness of
$\mathcal K_{|g|}$,
\[
 X=\bigcap_{h,g\in G}\Phi_{h,g}^{-1}(\mathcal K_{|g|})
\]
is closed in the compact metrizable space $\{0,1\}^G$, and is therefore
compact and metrizable. It is nonempty since
$\boldsymbol{0}\in X$.

If $\one_A\in X$ and $A'\subseteq A$, then
$
 I_{h,g}(A')\subseteq I_{h,g}(A)
$
for every $h,g\in G$. Since $\mathcal K_{|g|}$ is hereditary,
$I_{h,g}(A')\in\mathcal K_{|g|}$, so $\one_{A'}\in X$.

Finally, for $t\in G$ and $1_A\in X$, one has $t\one_A=\one_{tA}$, and $ I_{h,g}(tA)=I_{t^{-1}h,g}(A)\in \mathcal K_{|g|}$.
 Hence $tX\subseteq X$. Applying the same argument
to $t^{-1}$ gives $tX=X$. Thus $X$ is shift invariant, and therefore
a hereditary subshift.
\end{proof}

We next prove a uniform density property of $X$. Set
\[
 c:=\prod_{k=0}^{\infty}\frac{M_k}{M_k+1}.
\]
Since $\sum_{k=0}^{\infty}M_k^{-1}<\infty$, we have $c>0$.
\begin{lemma}\label{lem:nv-density}
For every nonempty finite set $F\subseteq G$, there exists a subset $J\subseteq F$
such that
\[
 |J|\geq c|F|
 \qquad\text{and}\qquad
 \one_J\in X.
\]
\end{lemma}
\begin{proof}
Denote
\[
 D:=\max_{w\in F}|w|.
\]

We first record the following elementary estimate. Given $n_1<\cdots<n_m$ and $M\geq1$, denote
\[
 r:=\left\lfloor\frac{m}{M+1}\right\rfloor+1,
 \qquad
 B:=\{n_r,\ldots,n_m\}.
\]
Then
\begin{equation}\label{eq:nv-tail-pruning}
 |B|=m-r+1\geq\frac{M}{M+1}m,
 \qquad
 |B|\leq Mr\leq Mn_r=M\min B.
\end{equation}

For $k\geq0$ and reduced words $u,v$ with $|v|=k$, set
\[
 L_k(u,v):=
 \{ua_nv:n\in\mathbb N,\ ua_nv\text{ is reduced}\}.
\]
As $(u,v)$ varies, these sets partition the reduced words of length
greater than $k$. Indeed, every such word admits a unique decomposition
$ua_nv,$ with $|v|=k.$

Starting with $F^{(0)}:=F$, we construct $F^{(k+1)}$ from $F^{(k)}$
for $k=0,\ldots,D-1$. For each $L_k(u,v)$ meeting $F^{(k)}$, write
\begin{equation}\label{eq:def F}
     F^{(k)}\cap L_k(u,v)
 =
 \{ua_{n_1}v,\ldots,ua_{n_m}v\},
 \qquad \text{with }n_1<\cdots<n_m,
\end{equation}
and set
\[
 r_k(u,v):=
 \left\lfloor\frac{m}{M_k+1}\right\rfloor+1.
\]
Define
\[
 F^{(k+1)}
 :=
 \{w\in F^{(k)}:|w|\leq k\}
 \cup
 \bigcup_{L_k(u,v)\cap F^{(k)}\neq\emptyset}
 \{ua_{n_j}v:r_k(u,v)\leq j\leq m\}.
\]
By \eqref{eq:nv-tail-pruning}, on each $L_k(u,v)$ the retained index
set belongs to $\mathcal S_k$ and has cardinality at least
\[
 \frac{M_k}{M_k+1}
 \bigl|F^{(k)}\cap L_k(u,v)\bigr|.
\]
Since the sets $L_k(u,v)$ are pairwise disjoint and partition the words
of length greater than $k$, it follows that
\[
 |F^{(k+1)}|
 \geq
 \frac{M_k}{M_k+1}|F^{(k)}|.
\]
Thus, setting
$
 J:=F^{(D)},
$
we obtain
\[
 |J|
 \geq
 \left(
 \prod_{k=0}^{D-1}\frac{M_k}{M_k+1}
 \right)|F|
 \geq c|F|.
\]

It remains to show that $\one_J\in X$. Fix
\[
 h=a_{i_1}\cdots a_{i_p},
 \qquad
 g=a_{j_1}\cdots a_{j_k}
\]
in reduced form, where $p,k\geq0$. Define
\[
 E(h,g)
 :=
 \begin{cases}
 \{i_p,j_1\}, & p,k\geq1,\\
 \{i_p\}, & p\geq1,\ k=0,\\
 \{j_1\}, & p=0,\ k\geq1,\\
 \emptyset, & p=k=0.
 \end{cases}
\]
Then $|E(h,g)|\leq2$. For every $n\notin E(h,g)$, the concatenation
$ha_ng$ is reduced and belongs to $L_k(h,g)$.

Suppose first that $k<D$. At stage $k$, let
\[
 B:=\{n_r,\ldots,n_m\}
\]
be the index set retained from $F^{(k)}\cap L_k(h,g)$ in the
construction of $F^{(k+1)}$ in \eqref{eq:def F}. By construction,
$B\in\mathcal S_k$. Since $J\subseteq F^{(k+1)}$, every
$n\in I_{h,g}(J)\setminus E(h,g)$ belongs to $B$. Hence
\[
 I_{h,g}(J)\setminus E(h,g)\subseteq B.
\]
The heredity of $\mathcal S_k$ therefore gives
\[
 I_{h,g}(J)\setminus E(h,g)\in\mathcal S_k.
\]
Moreover,
\[
 I_{h,g}(J)\cap E(h,g)\in\mathcal E_2.
\]
Hence
\[
 I_{h,g}(J)
 =
 \bigl(I_{h,g}(J)\setminus E(h,g)\bigr)
 \cup
 \bigl(I_{h,g}(J)\cap E(h,g)\bigr)
 \in\mathcal K_k.
\]

Now suppose that $k\geq D$. If $n\notin E(h,g)$, then $ha_ng$ is
reduced and
\[
 |ha_ng|=p+1+k>D.
\]
Since $J\subseteq F$ and every element of $F$ has length at most $D$,
we have $ha_ng\notin J$. Thus
$
 I_{h,g}(J)\subseteq E(h,g),
$
and hence $$I_{h,g}(J)\in\mathcal K_k.$$

Therefore
\[
 I_{h,g}(J)\in\mathcal K_{|g|}
\]
for every $h,g\in G$, and so $\one_J\in X$.
\end{proof}

\begin{proof}[Proof of Theorem~\ref{thm:naive-counterexample}]
By Proposition~\ref{prop:nv-subshift}, $X$ is a nonempty compact
metrizable hereditary subshift. We will prove that $X$ satisfies 
\begin{equation}\label{eq:positive entropy}
 h_{\mathrm{nv}}^{\mathrm{top}}(X,G)>0    
\end{equation}
and, for every compatible metric $\rho$ on $X$ and all $x,x'\in X$,
\begin{equation}\label{eq:asy.}
     \lim_{n\to\infty}\rho(a_nx,a_nx')=0.
\end{equation}

Let
\[
 \mathcal U:=\{U_0,U_1\},
 \qquad
 U_i:=\{x\in X:x(e)=i\},\quad i=0,1.
\]
Fix a nonempty finite subset $F\subset G$.
Applying Lemma~\ref{lem:nv-density} to $F^{-1}$, we obtain
$J\subseteq F^{-1}$ such that
\[
 |J|\geq c|F|
 \qquad\text{and}\qquad
 \one_J\in X.
\]
Since $X$ is hereditary, $\one_P\in X$ for every $P\subseteq J$.
Thus all binary patterns on $J$ occur in $X$, and hence
\[
 \log N(\mathcal U^F)
 \geq |J|\log2
 \geq c|F|\log2.
\]
Therefore
\[
 h_{\mathrm{nv}}^{\mathrm{top}}(X,G)
 \geq c\log2>0,
\]
which implies \eqref{eq:positive entropy}.

We next show \eqref{eq:asy.}. Fix $x=\one_A\in X$ and
$q\in G$. Since $a_n^{-1}=a_n$,
\[
 \{n\in\mathbb N:(a_nx)(q)=1\}
 =
 I_{e,q}(A)
 \in\mathcal K_{|q|}.
\]
By Proposition~\ref{prop:nv-families}, every member of $\mathcal K_{|q|}$ is finite, so $I_{e,q}(A)$
is finite.  Hence every
coordinate of $a_nx$ is eventually zero. Consequently, for every
compatible metric $\rho$ on $X$ and all $x,x'\in X$,
\[
 \rho(a_nx,a_nx')
 \leq
 \rho(a_nx,\boldsymbol{0})
 +
 \rho(\boldsymbol{0},a_nx')
 \to0,\qquad\text{as }n\to\infty.
\]
Since the generators $a_n$ are pairwise distinct, $(a_n)_{n\geq1}$ is
injective. In particular, there is no Li--Yorke pair along this
sequence.
\end{proof}

\providecommand{\bysame}{\leavevmode\hbox to3em{\hrulefill}\thinspace}
\providecommand{\MR}{\relax\ifhmode\unskip\space\fi MR }
\providecommand{\MRhref}[2]{%
	\href{http://www.ams.org/mathscinet-getitem?mr=#1}{#2}
}
\providecommand{\href}[2]{#2}

\end{document}